\documentclass[11pt]{amsart}
\advance\hoffset by -0.5 truecm
\usepackage{amsmath,amscd}
\usepackage{amssymb}
\usepackage{amsthm}
\usepackage{array}
\usepackage{graphicx}
\usepackage{color}

\usepackage[colorinlistoftodos]{todonotes}

\usepackage[T1]{fontenc}
\usepackage[utf8]{inputenc}
\usepackage[french,english]{babel}

\makeatletter
\newenvironment{abstracts}{%
	\ifx\maketitle\relax
	\ClassWarning{\@classname}{Abstract should precede
		\protect\maketitle\space in AMS document classes; reported}%
	\fi
	\global\setbox\abstractbox=\vtop \bgroup
	\normalfont\Small
	\list{}{\labelwidth\z@
		\leftmargin3pc \rightmargin\leftmargin
		\listparindent\normalparindent \itemindent\z@
		\parsep\z@ \@plus\p@
		
		\itemsep\medskipamount
	}%
}{%
	\endlist\egroup
	\ifx\@setabstract\relax \@setabstracta \fi
}

\newcommand{\abstractin}[1]{%
	\otherlanguage{#1}%
	\item[\hskip\labelsep\scshape\abstractname.]%
}
\makeatother

\usepackage{mathrsfs}
\usepackage{hyperref}

\usepackage{comment}

\usepackage{calc}
{\begin{list}{\arabic{enumi}.}{\usecounter{enumi}%
			\setlength{\labelsep}{0.5em}%
			\settowidth{\labelwidth}{\arabic{enumi}.}%
			\setlength{\leftmargin}{\labelwidth+\labelsep}}}%
	{\end{list}}

\theoremstyle{plain}
\newtheorem{thm}{Theorem}[section]
\newtheorem{con}{Conjecture}
\newtheorem*{thm*}{Theorem}
\newtheorem{cor}[thm]{Corollary}
\newtheorem*{cor*}{Corollary}
\newtheorem{lem}[thm]{Lemma}
\newtheorem*{lem*}{Lemma}

\newtheorem*{prop*}{Proposition}

\theoremstyle{definition}

\newtheorem*{defn*}{Definition}
\newtheorem{rem}[thm]{Remark}

\newtheorem*{Ex*}{Example}
\newtheorem*{question}{Question}

\DeclareRobustCommand{\SkipTocEntry}[5]{}

\newcounter{sidenote}
\usepackage{mathtools}

\usepackage{amsmath, amsthm, amssymb, enumerate, hyperref, xcolor, tikz-cd, listings,hyperref,comment}

\usepackage[tt=false, type1=true]{libertine}
\usepackage[varqu]{zi4}
\usepackage[libertine]{newtxmath}
\usepackage[T1]{fontenc}
\usepackage{parskip}

\renewcommand{\epsilon}{\varepsilon}

\newcommand{\re}{\text{Re}}
\newcommand{\C}{\mathbb{C}}

\newcommand{\K}{\mathbb{K}}

\newcommand{\R}{\mathbb{R}}

\newcommand{\Z}{\mathbb{Z}}

\renewcommand{\phi}{\varphi}

\def\semicolon{;}
\def\applytolist#1{
	\expandafter\def\csname multi#1\endcsname##1{
		\def\multiack{##1}\ifx\multiack\semicolon
		\def\next{\relax}
		\else
		\csname #1\endcsname{##1}
		\def\next{\csname multi#1\endcsname}
		\fi
		\next}
	\csname multi#1\endcsname}

\def\calc#1{\expandafter\def\csname c#1\endcsname{{\mathcal #1}}}
\applytolist{calc}QWERTYUIOPLKJHGFDSAZXCVBNM;
\def\bbc#1{\expandafter\def\csname bb#1\endcsname{{\mathbb #1}}}
\applytolist{bbc}QWERTYUIOPLKJHGFDSAZXCVBNM;
\def\bfc#1{\expandafter\def\csname bf#1\endcsname{{\mathbf #1}}}
\applytolist{bfc}QWERTYUIOPLKJHGFDSAZXCVBNM;
\def\sfc#1{\expandafter\def\csname s#1\endcsname{{\sf #1}}}
\applytolist{sfc}QWERTYUIOPLKJHGFDSAZXCVBNM;
\def\fc#1{\expandafter\def\csname f#1\endcsname{{\mathfrak #1}}}
\applytolist{fc}QWERTYUIOPLKJHGFDSAZXCVBNM;

\begin{document}

	\title[Deformative Magnetic Marked Length Spectrum Rigidity]{Deformative Magnetic Marked Length Spectrum Rigidity}
	
	\author[J. Marshall Reber]{James Marshall Reber}
	\address{100 Math Tower, 231 W 18th Ave, Columbus, OH 43210}
	\email{\href{mailto:marshallreber.1@osu.edu}{marshallreber.1@osu.edu}}
	
	
	
	\begin{abstracts}
		\abstractin{english}Let $M$ be a closed surface, $\{g_s \ | \ s \in (-\epsilon, \epsilon)\}$ be a smooth family of Riemannian metrics on $M$, and let $\{\lambda_s : M \rightarrow \R \ | \ s \in (-\epsilon, \epsilon)\}$ be a smooth family of smooth functions on $M$. 
		We show that if the magnetic curvature of each $(g_s, \lambda_s)$ is negative, the lengths of each periodic orbit remains constant as the parameter $s$ varies, and $\text{Area}(g_s) = \text{Area}(g_0)$, then there exists a smooth family of diffeomorphisms $\{f_s : M \rightarrow M \ | \ s \in (-\epsilon, \epsilon)\}$ such that $f_s^*(g_s) = g_0$ and $f_s^*(\lambda_s) = \lambda_0$. This generalizes a result of Guillemin and Kazhdan \cite{GK} to the setting of magnetic flows.
	\end{abstracts}

	\maketitle
	

	\section{Introduction} \label{sec:intro}
	
	\addtocontents{toc}{\SkipTocEntry}
	\subsection*{Motivation and Main Results}
	If $M$ is a closed oriented surface with Riemannian metric $g$ and $\kappa \in C^\infty(M)$, then the \emph{magnetic flow} generated by the pair $(g,\kappa)$ is the flow on the unit tangent bundle $S_gM$ determined by the equation
	\begin{equation} \label{eqn:magdiff}
		\frac{D\dot{\gamma}}{dt} = (\kappa \circ \gamma) i \dot{\gamma},
	\end{equation}
	where $i$ is the almost complex structure given by a rotation by $\pi/2$ according to the orientation. We refer to the smooth function $\kappa$ as the \emph{magnetic intensity} and the pair $(g,\kappa)$ as the \emph{magnetic system}. Solutions to Equation \eqref{eqn:magdiff} are called \emph{magnetic geodesics} for the magnetic system $(g,\kappa)$. The \emph{magnetic curvature} of the magnetic flow generated by $(g, \kappa)$ is given by
	\[ \K \coloneqq K - X^\perp(\kappa) + \kappa^2,\]
	where $X^\perp$ is the horizontal vector field and $K$ is the Gaussian curvature.
	Our goal is to prove the following result.
	
		\begin{thm} \label{thm:magGK}
		Let $M$ be a closed oriented surface, $\{g_s \ | \ s \in (-\epsilon, \epsilon)\}$ a smooth family of Riemannian metrics on $M$, and $\{\kappa_s : M \rightarrow \R \ | \ s \in (-\epsilon, \epsilon)\}$ a smooth family of smooth functions on $M$. Suppose that for every $s \in (-\epsilon, \epsilon)$ we have that $\K_s < 0$, where $\K_s$ is the magnetic curvature of the magnetic flow generated by $(g_s, \kappa_s)$. If the lengths of corresponding periodic orbits of $(g_s, \kappa_s)$ and $(g_0, \kappa_0)$ are the same and $\text{Area}(g_s) = \text{Area}(g_0)$ for all $s \in (-\epsilon, \epsilon)$, then there exists a smooth family of diffeomorphisms $\{f_s : M \rightarrow M \ | \ s \in (-\epsilon, \epsilon)\}$ satisfying $f_s^*(g_s) = g_0$ and $f_s^*(\kappa_s) = \kappa_0$.
	\end{thm}

	
	\begin{rem} \label{rem:conjugacy}
		Note that $\K < 0$ implies that the corresponding magnetic flow is Anosov \cite{W}. Since the magnetic flows are Anosov for each $s \in (-\epsilon, \epsilon)$, we have that each periodic orbit admits a well-defined continuation for all $s \in (-\epsilon, \epsilon)$ whose length is a well-defined smooth function of $s$. We are assuming that this function along with the action is constant.
	\end{rem}
	
	Theorem \ref{thm:magGK} is related to the marked length spectrum rigidity conjecture. Recall that if $(M,g)$ is a closed Riemannian manifold with negative sectional curvature, then inside of every free homotopy class there is a unique closed geodesic for $g$. The \emph{marked length spectrum} is defined to be the function which takes a free homotopy class and returns the length of the unique closed geodesic inside of it. If a magnetic flow has negative magnetic curvature, then the marked length spectrum for the magnetic flow can be defined analogously. The following conjecture is well-known.
	
	\begin{con}[\cite{burnskatok}] \label{con}
		Let $M$ be a closed manifold. If $g$ and $g'$ be two negatively curved metrics on $M$ with the same marked length spectrum, then there is a diffeomorphism $f : M \rightarrow M$ such that $f^*(g) = g'$.
	\end{con}
	
	It was shown in \cite{GK} that if $M$ is a closed surface, then the conjecture holds provided the metrics can be connected by a smooth path of metrics with negative curvature along which the length spectrum is the same. Theorem \ref{thm:magGK} can be seen as the magnetic generalization -- if we can connect two negatively curved magnetic flows $(g, \kappa)$ and $(g', \kappa')$ by a path of negatively curved magnetic flows along which the marked length spectrum is constant, then there exists a diffeomorphism $f : M \rightarrow M$ so that $f^*(g) = g'$ and $f^*(\kappa) = \kappa'$.
	
	To the author's best knowledge, the only other progress towards a magnetic version of marked length spectrum rigidity can be found in \cite{grognet}. Adapting the arguments in \cite[Th\'{e}or\`{e}me 7.3]{grognet}, one can show that if a negatively curved magnetic flow shares the same marked length spectrum as a geodesic flow and the corresponding metrics have the same area, then the magnetic flow must be a geodesic flow and the metrics must be isometric. This result, along with Theorem \ref{thm:magGK}, leads us to the following question.
	
	\begin{question}
		Let $M$ be a closed oriented surface and let $(g,\kappa)$ and $(g',\kappa')$ be two magnetic flows with negative magnetic curvature and with the same marked length spectrum and $\text{Area}(g) = \text{Area}(g')$. Does there exist a diffeomorphism $f : M \rightarrow M$ so that $f^*(g) = g'$ and $f^*(\kappa) = \kappa'$?
	\end{question}

	We break up the proof of Theorem \ref{thm:magGK} into two steps. First we construct a smooth family of isometries $\{f_s : M \rightarrow M\}$ following the scheme of \cite{GK}. Note that we cannot directly use their arguments due to the magnetic intensities, and so appropriate modifications are made along the way. After constructing the isometries, we are able to reduce the problem to considering a family of magnetic systems $\{(g, \kappa_s) \ | \ s \in (-\epsilon, \epsilon)\}$ which all share a common metric $g$. The final step is to show that, in this setting, we must have $\smash{\frac{d}{ds}} \kappa_s = 0$.

	\addtocontents{toc}{\SkipTocEntry}
	\subsection*{Organization}
	The paper is organized as follows.
	\begin{itemize}
		\item In Section \ref{sec:preliminaries} we review the geometry of $S_gM$, the definition of a magnetic flow, Cartan's structural equations for magnetic flows, the Fourier decomposition of $S_gM$ following \cite{GK}.
		\item In Section \ref{subsec:outline}  we outline the proof of Theorem \ref{thm:magGK}, giving the argument without the details.
		\item In Section \ref{sec:proof11} we fill in the details of the proof of Theorem \ref{thm:magGK}.
	\end{itemize}
	
	\addtocontents{toc}{\SkipTocEntry}
	\subsection*{Acknowledgements}
	I would like to thank Andrey Gogolev for his advice and support throughout the project. I would also like to thank Gabriel Paternain and Javier Echevarr\'{i}a Cuesta for pointing out mistakes in earlier drafts.
	
	\section{Preliminaries} \label{sec:preliminaries}
	\addtocontents{toc}{\SkipTocEntry}
	\subsection{Geometry of $S_gM$}
	
	Let $M$ be a closed oriented surface. Given a Riemannian metric $g$, we denote the unit tangent bundle with respect to $g$ by $S_gM$, and we denote the footprint map by $\pi : TM \rightarrow M$. Since the manifold is oriented, we have an $S^1$-action on $S_gM$ given by rotation. We define the rotation flow by
	\[ \rho^t(x,v) \coloneqq (x, e^{it}v).\]
	The infinitesimal generator for the rotation flow is the \emph{vertical vector field}, denoted by $V$. If we let $g^t$ be the geodesic flow associated to $g$, then the infinitesimal generator for this flow is the \emph{geodesic vector field}, denoted by $X$. Finally, if we define the curve
	\[ \gamma_{(x,v)}(t) \coloneqq \pi \circ g^t(x,v),\]
	then the horizontal flow is given by
	\[ h^t(x,v) \coloneqq (\gamma_{(x,iv)}(t), Z(t)),\]
	where $Z(t)$ is the parallel transport of $v$ along $\gamma_{(x,iv)}(t)$. The infinitesimal generator for this flow is the \emph{horizontal vector field}, denoted by $X^\perp$. The vector fields $\{X, X^\perp, V\}$ give us a moving frame on $S_gM$ called \emph{Cartan's moving frame}. Dual to these vector fields are $1$-forms $\{\alpha, \beta, \psi\}$ on $S_gM$. Following \cite[Section 7.2]{singer} and \cite[Section 7]{MP}, we have \emph{Cartan's structural equations}:
	\begin{equation} \label{eqn:cartan} \begin{gathered}
			[V,X] = X^\perp, \ \ [V, X^\perp] = -X, \ \ [X, X^\perp] = KV, \\
			d\alpha = \psi \wedge \beta, \ \ d\beta = -\psi \wedge \alpha, \ \ d\psi = -(K \circ \pi) \alpha \wedge \beta,
		\end{gathered}
	\end{equation}
	where $K$ is the Gaussian curvature of $(M,g)$. Let $\Sigma \coloneqq -\alpha \wedge d\alpha = \alpha \wedge \beta \wedge \psi$ be a volume form on $S_gM$ and let $\mu$ be the corresponding Liouville measure. Finally, using \cite{plante}, we observe that if the unit tangent bundle $S_gM$ admits an Anosov flow, then the genus of $M$ must be at least two. The Gysin sequence \cite{bott} allows us to deduce the following.
	
	\begin{thm}[{\cite[Corollary 8.10]{MP}}] \label{thm:gysin}
		Let $\pi : S_gM \rightarrow M$ denote the footprint map restricted to the unit tangent bundle. If there exists a magnetic system $(g,\kappa)$ on $M$ so that the corresponding magnetic flow is Anosov, then $\pi^* : H^1(M,\R) \rightarrow H^1(S_gM, \R)$ is an isomorphism.
	\end{thm}

	\addtocontents{toc}{\SkipTocEntry}
	\subsection{Magnetic Flows}
	
	As mentioned in Section \ref{sec:intro}, if $g$ is a Riemannian metric and $\kappa \in C^\infty(M, \R)$, then we can associate a magnetic flow to the pair $(g,\kappa)$ by considering solutions to Equation \eqref{eqn:magdiff}. Observe that solutions to Equation \eqref{eqn:magdiff} correspond to closed curves which have geodesic curvature $b$ \cite{grognet}.

	Recall that associated to $g$ we have the \emph{area form} on $M$ defined by
	\[ (\Omega_a)_x(v,w) \coloneqq g_x(iv, w).\]
	Given any closed $2$-form $\sigma$ on $M$, there exists a $\kappa \in C^\infty(M, \R)$ such that $\sigma = \kappa \Omega_a.$ 
	This gives us a correspondence between smooth functions on $M$ and closed $2$-forms on $M$. If we define $H : TM \rightarrow \R$ by
	\[ H(x,v) \coloneqq \frac{1}{2}\|v\|_x^2,\]
	then the magnetic flow associated to the pair $(g, \kappa)$ is generated by the vector field $F$ satisfying
	\[ \iota_F(-d\alpha + \pi^*(\sigma)) = dH.\]
	Note that $\omega \coloneqq - d\alpha + \pi^*(\sigma)$ is a symplectic form, which we refer to as the \emph{magnetic symplectic form}. Hence, the magnetic flow is a Hamiltonian flow with respect to the above Hamiltonian and the magnetic symplectic form. The following variational observation will be useful throughout.
	
	\begin{lem}[{\cite[Lemma 4.1]{DP}}] \label{lem:minimizers}
	Let $(g, \kappa)$ be a magnetic system, and let $\gamma_s : [0,T_s] \rightarrow M$ be a smooth family of smooth closed curves, with $\gamma_0 \eqqcolon \gamma$ a closed magnetic geodesic for the magnetic system $(g,\kappa)$. If $S$ is the variational vector field along $\gamma$, then
	\[ \frac{d}{ds}\Big|_{s=0}\int_0^{T_s} H(\gamma_s(t)) dt = \int_0^T \kappa(\gamma(t)) \Omega_a(\dot{\gamma}(t), S(t)) dt. \]
	\end{lem}
	
	Using \cite[Lemma 7.7]{MP}, we see that the infinitesimal generator of the magnetic flow $F$ is of the form $F = X + \kappa V$. Furthermore, notice that we can write 
	\begin{equation} \label{eqn:primitive} \kappa \Omega_a = \sigma = c K\Omega_a + d \theta,\end{equation}
	where $\theta$ is a $1$-form on $M$ and 
	\[ c \coloneqq \frac{1}{2\pi \chi(M)} \int_M \kappa \Omega_a.\]
	With the aid of Equation \eqref{eqn:cartan}, we observe that if one restricts the magnetic symplectic form to $S_gM$, then we have 
	\begin{equation} \label{eqn:primitive2} \omega = d(-\alpha - c \psi + \pi^*(\theta)).\end{equation} 
	It is also easy to see that we have
	\begin{equation} \label{eqn:identity_symplectic} \iota_F \Sigma = \beta \wedge \psi + \kappa (\alpha \wedge \beta) = \omega.\end{equation} These facts together show that the magnetic flow on $S_gM$ is homologically full, i.e.\ every integral homology class has a magnetic geodesic representative \cite[Lemma 7.1]{Solly}.\footnote{One could also deduce this fact using \cite{ghys}.} This property is relevant due to the recent \emph{abelian Livshits theorem}, proven by Gogolev and Rodriguez Hertz in \cite{gogolevhertz0}. We state the result here in the language of magnetic flows for the readers convenience.
	
	\begin{thm} \label{thm:abelian_livshits}
		Suppose that the magnetic flow associated to the magnetic system $(g,\kappa)$ is Anosov. If $\varphi : S_gM \rightarrow \R$ is a smooth function such that 
		\begin{equation} \label{eqn:hom_triv_per_obs} \int_\gamma \varphi = 0 \end{equation}
		for every homologically trivial closed orbit $\gamma$ for the magnetic flow,  then there is a closed $1$-form $\omega$ on $M$ along with $u \in C^\infty(S_gM,\R)$ so that $\varphi =\omega + F(u)$, where $F$ is the infinitesimal generator of the magnetic flow.
	\end{thm}
	
	\begin{proof}
		Using \cite[Theorem 3.3]{gogolevhertz0}, we can deduce that there is a closed $1$-form $\xi$ on $S_gM$ and a smooth function $w \in C^\infty(S_gM,\R)$ so that $\varphi = \xi(F) + F(w)$. Using Theorem \ref{thm:gysin}, we see that there is a closed $1$-form $\omega$ on $M$ along with $q \in C^\infty(S_gM,\R)$ so that $\xi = \pi^*(\omega) + dq$. Observe that for $v \in S_gM$, we have $d_v\pi(F(v)) = v$, hence contracting this equation with $F$ yields $\xi(F) = \omega + F(q)$. The result now follows by letting $u \coloneqq q + w$.
	\end{proof}
	
	We highlight one particularly useful application of Theorem \ref{thm:abelian_livshits}, which is a direct consequence of \cite[Theorem B]{DP}.
	
	\begin{cor} \label{cor:one-form}
		Suppose that the magnetic flow associated to the magnetic system $(g,\kappa)$ is Anosov. If $\theta$ is a $1$-form on $M$ which satisfies Equation \eqref{eqn:hom_triv_per_obs} when viewed as a function on $S_gM$, then $\theta$ is closed.
	\end{cor}

	Finally, using Equation \eqref{eqn:primitive2}, one can deduce the following.
	
	\begin{lem}[{\cite[Proposition 4.5]{javier}}] \label{lem:const_cohom}
		Let $(g_1, \kappa_1)$ and $(g_2, \kappa_2)$ be two magnetic systems such that $\text{Area}(g_1) = \text{Area}(g_2)$. Denoting the area form corresponding to the metric $g_i$ by $\Omega_{a,i}$. If the corresponding magnetic flows to the magnetic systems $(g_1, \kappa_1)$ and $(g_2, \kappa_2)$ are smoothly conjugate, then $[\kappa_1 \Omega_{a,1}] = \pm [\kappa_2 \Omega_{a,2}].$
	\end{lem}

	\addtocontents{toc}{\SkipTocEntry}
	\subsection{Fourier Analysis on $S_gM$}
	
	We define the following sesquilinear form on $L^2(S_gM, \C)$:
	\[ (u,v) \coloneqq \int_{S_gM} u \overline{v} d\mu.\]
	Consider the family of diffeomorphisms of $S_{g_0}M$ generated by $V$, i.e.\ the family $\{e^{i\theta}\}$. Associated to these diffeomorphisms are operators $U_\theta$ on $L^2(S_gM, \C)$ defined by
	\[ U_\theta(f) \coloneqq f \circ e^{i \theta}.\]
	Since the maps $e^{i \theta}$ are volume preserving, we have that the operators $U_\theta$ are unitary. We also note the operators $U_\theta$ are strongly continuous, in the sense that
	\[ \lim_{\theta \rightarrow \theta_0} U_\theta(f) = U_{\theta_0}(f).\]
	Thus we are able to use Stone's theorem \cite{stone} to extend $V$ to a self-adjoint densely defined operator on $L^2(S_gM, \C)$. We denote this extension by $-iV$.
	By \cite[Lemma 3.1]{GK}, the space $L^2(S_gM, \C)$ decomposes orthogonally as a direct sum of eigenspaces of $-iV$:
	\[ L^2(S_gM, \C) = \bigoplus_{k \in \Z} H_k, \text{ where } H_k \coloneqq \{ f \in L^2(S_gM, \C) \ | \ -iV f = kf\}.\]
	If we let $\Omega_k \coloneqq C^\infty(S_gM, \C) \cap H_k$, then we see that for all $u \in C^\infty(S_gM, \C)$ we have a Fourier expansion
	\[ u = \sum_{k=-\infty}^\infty u_k, \text{ where } u_k \in \Omega_k = \{f \in C^\infty(S_gM, \C) \ | \ Vf = ikf\}.\]
	Let $u \in C^\infty(S_gM, \C)$. If there exists an $N$ so that $u_k = 0$ for all $|k| > N$, then we say that $u$ has \emph{finite degree}. If $N$ is the smallest positive integer such that $u_k = 0$ for all $|k| > N$, then we say that $u$ has \emph{degree $N$}. 

	Following \cite[Section 3]{GK}, we define the following first order elliptic operators
	\[ \eta^{\pm} : C^\infty(S_gM, \C) \rightarrow C^\infty(S_gM, \C), \ \ \eta^{\pm} \coloneqq \frac{X \mp iX^\perp}{2}.\]
	
	We observe the following.
	\begin{enumerate}[(i)]
		\item We have $X = \eta^+ + \eta^-$ and $X^\perp = i \eta^+ - i \eta^-$.
		\item If $F$ is the vector field generating the magnetic flow given by $(g,\kappa)$, then $F = \eta^+ + \eta^- + \kappa V$. Moreover, we see that $\{\eta^+, \eta^-, V\}$ spans $S_gM$ at each point.
		\item Using Cartan's structural equations \eqref{eqn:cartan}, one can show that
		\[\eta^{\pm} : \Omega_k \rightarrow \Omega_{k \pm 1}.\]
		Thus, we see that $\eta^+$ raises the degree and $\eta^-$ lowers the degree.
	\end{enumerate}
	
	Throughout, we will be working with functions that have degree at most $2$. The next observation will give us a magnetic analogue of \cite[Theorem 3.6]{GK} for \emph{symmetric $2$-tensors}, i.e.\ functions $v \in \bigoplus_{|k| \leq 2} \Omega_k$ satisfying $\overline{v_k} = v_{-k}$ for each $k$.
	
	\begin{thm}[{\cite[Theorem 1.1]{A}}] \label{thm:ainsworth}
		Let $(g,\kappa)$ be a magnetic system such that the corresponding magnetic flow is Anosov. If $v$ is a symmetric $2$-tensor and $Xu = v$, then $u \in \bigoplus_{|k| \leq 1} \Omega_k$. 
	\end{thm}
	
	Finally, the following will tell us when solutions $u$ to the equation $Xu = v$ can be interpreted as a $1$-form.
	
	\begin{lem}[{\cite[Lemma 4.1]{GK}, \cite[Proof of Theorem 12.2]{MP}]}] \label{lem:one_form}
		Let $M$ be a closed surface, and let $g$ be a Riemannian metric on $M$ with everywhere negative curvature. Suppose $\beta \in \Omega_{-2} \oplus \Omega_0 \oplus \Omega_2$ satisfies the condition that $\overline{\beta_{-2}} = \beta_2$. If $X\delta = \beta$ with $\delta \in \Omega_{-1} \oplus \Omega_1$, then $\delta$ is a $1$-form.
	\end{lem}

	\section{Outline of the Proof of Theorem \ref{thm:magGK}} \label{subsec:outline}
	
	From here on, we denote with a subscript $s$ the corresponding object for the magnetic system $(g_s, \kappa_s)$. 	We start by constructing a smooth family of smooth conjugacies between the corresponding magnetic flows. These will be necessary for constructing the isometries.

	\begin{lem} \label{lem:smoothconjugacies}
		Let $M$ be a closed oriented surface, $\{g_s \ | \ s \in (-\epsilon, \epsilon)\}$ a smooth family of Riemannian metrics on $M$, $\{\kappa_s : M \rightarrow \R \ | \ s \in (-\epsilon, \epsilon)\}$ a smooth family of smooth functions on $M$. Suppose that for every $s \in (-\epsilon, \epsilon)$ we have that $\K_s < 0$, where $\K_s$ is the magnetic curvature of the magnetic flow generated by $(g_s, \kappa_s)$. If the lengths of corresponding periodic orbits of $(g_s, \kappa_s)$ and $(g_0, \kappa_0)$ are the same for each $s \in (-\epsilon, \epsilon)$, then we have a smooth family of smooth conjugacies $\{h_s : S_{g_0}M \rightarrow S_{g_s}M\}$ between the magnetic flows with $h_0 = \text{Id}$. Furthermore, if $\text{Area}(g_s) = \text{Area}(g_0)$ for all $s \in (-\epsilon, \epsilon)$, then $h_s^*(\Sigma_s) = \Sigma_0$.
	\end{lem}
	
	With the smooth conjugacies in hand, we can construct the isometries. Define the following family of symmetric $2$-tensors on $TM$:
	\[ \beta_t \coloneqq \frac{d}{ds}\Big|_{s=t} g_s , \ \ \beta \coloneqq \beta_0.\]
	Note that we are viewing $g_s$ as a function on $TM$ given by
	\[(x,v) \mapsto (g_s)_x(v,v) =: \|v\|_s^2.\]
	Using \cite[Lemma 4.1]{GK}, we can write $\beta = \beta_{-2} + \beta_0 + \beta_2$ with $\beta_k \in \Omega_k$ and $\overline{\beta_{-2}} =  \beta_2$. The next step is to utilize Theorem \ref{thm:abelian_livshits} to show that, up to a closed $1$-form, $\beta$ integrates to zero over closed orbits of the magnetic flow given by $(g_0, \kappa_0)$.
	\begin{lem} \label{lem:livsic}
		There exists a closed $1$-form $\xi$ on $M$ so that for every closed orbit $\gamma$ of the magnetic flow given by $(g_0, \kappa_0)$, we have
		\[ \int_{\gamma} [\beta + \xi] = 0.\]
		In particular, there is a smooth function $u \in C^\infty(S_{g_0}M, \R)$ so that
		\[ Fu = \beta + \xi.\]
	\end{lem}
	The above lemma along with Theorem \ref{thm:ainsworth} implies that $u$ has degree $1$. Furthermore, writing $\delta = u_{-1} + u_1$, we can rewrite $Fu = \beta + \xi$ as the following system of equations:
	\[ \begin{cases} X\delta = \beta, \\ Xu_0 + \kappa V \delta = \xi.\end{cases}\]
	Lemma \ref{lem:one_form} implies that $\delta$ is a $1$-form. Doing this procedure for every $s$, we get a corresponding family of $1$-forms $\delta_s$, and using \cite[Theorem 2.2]{LMM} we see that the $1$-forms vary smoothly with respect to $s$. Let $Z_s$ be the vector field dual to $\delta_s$ under the metric $g_s$. If we let $f_s$ be the smooth family of diffeomorphisms satisfying
	\[ Z_s = \frac{df_s}{ds} \circ f_s^{-1},\]
	then we see that $g_s$ and $g_s' \coloneqq f_s^*(g_0)$ satisfy the same differential equation with the same initial condition:
	\[ \beta_s = Z_s(g_s), \ \ \beta_s' = Z_s(g_s'), \text{ and } g_0 = g_0'\]
	By existence and uniqueness of solutions to differential equations, we must have that $g_s = g_s'$ for each $s \in (-\epsilon, \epsilon)$, and so we have constructed our family of isometries.
	
	As mentioned at the end Section \ref{sec:intro}, we can now reduce the problem using the isometries by considering the family of magnetic flows given by the metric $g_0$ and the smooth functions $(f_s^{-1})^*(\kappa_s)$. Let $\kappa_s' \coloneqq (f_s^{-1})^*(\kappa_s)$. The goal is to show that $\kappa_s'$ is constant with respect to $s$. To that end, observe that we can now write the family closed $2$-forms $\sigma_s$ associated to the magnetic system $(g_s, \kappa_s)$ as 
	\begin{equation} \label{eqn:family} \sigma_s = c K \Omega_a + d\theta_s,\end{equation}
	where $\theta_s$ is a family of $1$-forms on $M$. As deduced in the proof of \cite[Theorem 3.2.4]{McDuff}, we may assume that $\theta_s$ also varies smoothly in $s$. Let $\dot{\theta}_r \coloneqq \frac{d}{ds}|_{s=r} \theta_s$. Another application of Theorem \ref{thm:abelian_livshits} along with the Gauss-Bonnet theorem will yield the following.
	
	\begin{lem} \label{lem:livsic2}
		There exists a closed $1$-form $\eta$ on $M$ so that for every closed orbit $\gamma$ of the magnetic flow given by $(g_0, \kappa_0')$, we have 
		\[ \int_\gamma [\dot{\theta}_0 + \eta] = 0.\]
		In particular, $\dot{\theta}_0$ is a closed $1$-form on $M$.
	\end{lem}
	
	Note that the choice of $s=0$ was arbitrary, so this holds for all $s$. Taking a derivative of Equation \eqref{eqn:family} with respect to $s$ and using the fact that $\kappa_s' \Omega_a = \sigma_s$, we deduce that $\kappa_s'$ is constant in $s$, as desired.

	\section{Proof of Theorem \ref{thm:magGK}} \label{sec:proof11}
	
	We start by proving the existence of the smooth family of smooth conjugacies.
	
	\begin{proof}[Proof of Lemma \ref{lem:smoothconjugacies}]
		Using \cite[Theorem A.1]{LMM}, there exists a smooth family of orbit equivalences between the flows such that the lengths of corresponding closed orbits are the same. We use \cite[Theorem 6.3.9]{FH} and \cite[Theorem 2.2]{LMM} to upgrade each orbit equivalence to a $C^0$-conjugacy in such a way so that the family remains smooth. Finally, we use \cite[Theorem 1.2]{gogolevhertz} to get that the conjugating homeomorphisms are actually smooth.
		
		Denote the smooth family of smooth conjugacies by $h_s : S_{g_0}M \rightarrow S_{g_s}M$. Notice that there exists a smooth function $J_s \in C^\infty(S_{g_0}M, \R)$ so that $h_s^*(\Sigma_s) = J_s \Sigma_0$. Since the magnetic flow preserves the volume form, we see that 
	\[ J_s \Sigma_0 = h_s^*(\Sigma_s) = (\varphi_s^t \circ h_s)^*(\Sigma_s) = (h_s \circ \varphi_0^t)^*(\Sigma_s) = (J_s \circ \varphi_0^t) \Sigma_0.\]
	We deduce that $J_s$ is constant using the fact that $(g_0, \kappa_0)$ has a dense orbit. Furthermore, a change of variables argument shows that $J_s = \text{Area}(g_s)/\text{Area}(g_0)$, so the area assumption yields that $J_s \equiv 1$.
	\end{proof}
	
	Following Section \ref{subsec:outline}, we write
	\[ \beta_t \coloneqq \frac{d}{ds}\Big|_{s=t} g_s , \ \ \beta \coloneqq \beta_0.\]
	We now prove Lemma \ref{lem:livsic}.

	\begin{proof}[Proof of Lemma \ref{lem:livsic}]
		To help with notation, let $SM$ be the principal circle bundle over $M$ with fibers given by $S_xM = (T_xM \setminus \{0\})/\sim$, where $v \sim w$ if and only if $v = C w$ with $C > 0$. For each metric $g_s$, there is a bundle isomorphism $\zeta_s : S_{g_s}M \rightarrow SM$ which is defined by sending a vector to its equivalence class. Using the maps $\zeta_s$, we push all of the forms and flow to $SM$ and work on this common bundle; abusing notation, we use the same symbol to denote the corresponding object from $S_{g_s}M$ on $SM$. 
		
		Let $p : SM \rightarrow M$ be the projection map. As we saw in Section \ref{sec:preliminaries}, we have 
		\[\sigma_s = c_s K_s (\Omega_{a})_s + d\theta_s.\]
		Note that Lemma \ref{lem:const_cohom} yields $c_s = c_0 \eqqcolon c$ for every $s$, thus $p^*(\sigma_s) = d(-c \psi_s + p^*(\theta_s))$. Define $\tau_s \coloneqq -\alpha_s - c \psi_s + p^*(\theta_s)$, so $d\tau_s = \omega_s$ on $SM$. Furthermore, using Lemma \ref{lem:smoothconjugacies} and Equation \eqref{eqn:identity_symplectic}, we see that $h_s^*(\omega_s) = \omega_0$, and thus $h_s^*(\tau_s)-\tau_0$ is a smooth family of closed $1$-forms on $SM$. 
		
		Let $\gamma_0$ be a closed homologically trivial orbit for $(g_0, \kappa_0)$ with length $T$, and let $\gamma_s$ be the corresponding orbits for $(g_s, \kappa_s)$. Since $h_s^*(\tau_s) - \tau_0$ is closed, we have 
		\[ \int_{\gamma_0} \iota_{F_0}[h_s^*(\tau_s) - \tau_0] = 0, \]
		and using the fact that $h_s$ is a smooth conjugacy, we deduce 
		\begin{equation} \label{eqn:identity_conjugacy} \int_{\gamma_s} \iota_{F_s} \tau_s = \int_{\gamma_0} \iota_{F_0} \tau_0 \text{ for all } s \in (-\epsilon, \epsilon). \end{equation}
		Consider the parameterization given by 
		\[ \Gamma : [0,s] \times [0,T] \rightarrow SM, \quad \Gamma(s,t) \coloneqq (\gamma_s(t), \dot{\gamma}_s(t)).\]
		Denote the image of this parameterization by 
		\[B_s \coloneqq \{(\gamma_s(t), \dot{\gamma}_s(t)) \ | \ 0 \leq r \leq s, 0 \leq t \leq T\} \subseteq SM.\] 
		For simplicity, we write $F(r,t) \coloneqq F_r(\gamma_r(t), \dot{\gamma}_r(t))$ and $W(r,t) \coloneqq \frac{d}{ds} |_{s=r} (\gamma_s(t), \dot{\gamma}_s(t))$. With this, we define $\tau$ and $\alpha$ to be $1$-forms on $B_s$ satisfying $\tau(W) = 0 = \alpha(W)$, $\tau(F_r) = \tau_r(F_r)$, and $\alpha(F_r) = \alpha_r(F_r)$. In other words, these are the $1$-forms which ignore the variational direction, and along each orbit behave like the corresponding $1$-form. Using Stokes' theorem along with the fact that $\gamma_0$ is homologically trivial, we observe that 
		\[ 0 = \int_{B_s} d\tau = \int_0^s \int_0^T (\Gamma^*(d\tau))_{(r,t)}\left( \frac{d}{dt}, \frac{d}{dr}\right) dt dr = -\int_0^s \int_0^T  W(r,t)((\tau_r)_{(\gamma_r(t), \dot{\gamma}_r(t))}(F(r,t)) dt dr. \]
		On the other hand, observe that 
		\[ (d \tau)_{(\gamma_r(t), \dot{\gamma}_r(t))} (F(r,t), W(r,t)) = F(r,t) ((\tau_r)_{(\gamma_r(t), \dot{\gamma}_r(t))}(W(r,t))) + (d \tau_r)_{\gamma_r(t), \dot{\gamma}_r(t)} (F(r,t), W(r,t)),\]
		so
		\[ 0 = \int_{B_s} d\tau = \int_0^s \int_0^T (d \tau_r)_{(\gamma_r(t), \dot{\gamma}_r(t))} (F(r,t), W(r,t))dt dr.\]
		Following the same argument with $\alpha$ in place of $\tau$ and using the length assumption, we are left with
		\[ 0 =  \int_0^s \int_0^T p^*(\kappa_r (\Omega_a)_r)(F(r,t),W(r,t))dtdr.\]
		Taking the derivative of both sides with respect to $s$ and using Lemma \ref{lem:minimizers} along with the length assumption, we have $\dot{H}_0$ integrates to zero along every homologically trivial orbit. The result now follows by Theorem \ref{thm:abelian_livshits}.
	\end{proof}
	
	
	As mentioned in Section \ref{subsec:outline}, this was the missing ingredient needed for us to get our isometries. We now have a smooth family of diffeomorphisms $f_s : M \rightarrow M$ such that $f_s^*(g_0) = g_s$. We switch our focus to the family $\{(g_0, \kappa_s') \ | \ s \in (-\epsilon, \epsilon)\}$ where $\kappa_s' \coloneqq (f_s^{-1})^*(\kappa_s)$. We now prove Lemma \ref{lem:livsic2}.
	
	\begin{proof}[Proof of Lemma \ref{lem:livsic2}]
		Let $\gamma_0$ be a closed orbit for $(g_0, \kappa_0')$ which is homologically trivial and has length $T$, let $\gamma_s$ be the corresponding closed orbits for $(g_s, \kappa_s')$, and let 
		\[\hat{B}_s \coloneqq \{\gamma_r(t) \ | \ 0 \leq r \leq s, 0 \leq t \leq T\} \subseteq M\] 
		be the band swept out by these curves. Using Lemma \ref{lem:minimizers} along with the Gauss-Bonnet theorem, we deduce that 
		\[ 0 = \frac{d}{ds} \Big|_{s=0} \int_{\hat{B}_s} \sigma_0 = - c \frac{d}{ds} \Big|_{s=0} \int_{\gamma_s} \kappa_s + \frac{d}{ds} \Big|_{s=0} \int_{\gamma_s} \theta_0.\]
		On the other hand, notice that $\iota_{F_s} \tau_s = -1 - c \kappa_s' + \theta_s$. The length assumption along with Equation \eqref{eqn:identity_conjugacy} implies that 
		\[ 0 = - c \frac{d}{ds} \Big|_{s=0} \int_{\gamma_s} \kappa_s' + \frac{d}{ds} \Big|_{s=0} \int_{\gamma_s} \theta_s.\]
		Combining these observations yields that for every homologically trivial closed orbit $\gamma_0$ for $(g_0, \kappa_0')$, we have 
		\[ \int_{\gamma_0} \dot{\theta}_0 = 0,\]
		where the dot indicates derivative with respect to $s$.
		Using Corollary \ref{cor:one-form} and the fact that $s=0$ was arbitrary, we have that $\dot{\theta}_s$ is a closed form for each $s$. 
	\end{proof}
	
	\bibliographystyle{alpha}

\end{document}